\documentclass[12pt,a4paper,dvips]{amsart}
\usepackage[text={16cm,24.5cm}]{geometry}

\usepackage{amsfonts, amsmath, amssymb, amsthm}
\usepackage[mathscr]{eucal} 
\usepackage{hyperref}

\usepackage{graphicx,color}
\usepackage{booktabs}
\usepackage{bbm}
\usepackage{xspace}

\providecommand{\KK}{{\mathbb{K}}}

\providecommand{\RR}{{\mathbb{R}}}

\providecommand{\cH}{{\mathscr{H}}}
\providecommand{\cM}{{\mathscr{M}}}

\providecommand{\cP}{{\mathscr{P}}}

\providecommand{\cU}{{\mathscr{U}}}

\providecommand{\vones}{\mathbbm{1}}

\providecommand{\SetOf}[2]{\left\{#1\vphantom{#2}\,\right.\left|\,\vphantom{#1}#2\right\}}
\providecommand{\smallSetOf}[2]{\{#1\,|\,#2\}}
\providecommand{\bigSetOf}[2]{\bigl\{#1\,\big|\,#2\bigr\}}
\providecommand{\GeneratedBy}[2]{\left\langle#1\vphantom{#2}\,\right.\left|\,\vphantom{#1}#2\right\rangle}

\DeclareMathOperator{\conv}{conv}

\DeclareMathOperator{\Sym}{Sym}

\DeclareMathOperator{\initial}{in}

\providecommand\polymake{\texttt{polymake}\xspace}

\providecommand\cprime{$'$}

\providecommand{\TA}{{\mathbb{T}}}
\providecommand{\TP}{{\mathbb{TP}}}
\providecommand{\trop}{{\rm trop}}
\DeclareMathOperator{\tconv}{tconv}
\DeclareMathOperator{\tdet}{tdet}
\DeclareMathOperator{\tsgn}{tsgn}

\DeclareMathOperator{\type}{type}
\providecommand{\MatroidPolytope}[1]{P_{#1}}

{\end{list}}

\theoremstyle{plain}
\newtheorem{theorem}{Theorem}
\newtheorem{proposition}[theorem]{Proposition}

\newtheorem{lemma}[theorem]{Lemma}

\theoremstyle{definition}

\newtheorem{example}[theorem]{Example}

\newtheorem{remark}[theorem]{Remark}

\usepackage[boxed,vlined]{algorithm2e}
\SetKwInOut{Input}{input}
\SetKwInOut{Output}{output}
\SetArgSty{rm}
\SetFuncSty{tt}
\setalcapskip{1ex}
 
\renewcommand\thealgocf{\Alph{algocf}}

\graphicspath{ {mpost/} }

\begin{document}

\title{Tropical Convex Hull Computations}

\author{Michael Joswig}
\address{Fachbereich Mathematik, TU Darmstadt, and Institut f\"ur Mathematik, TU Berlin, Germany}
\thanks{This research is supported by DFG Research Unit ``Polyhedral Surfaces''.}
\email{joswig@mathematik.tu-darmstadt.de}

\date{\today}

\begin{abstract}
  This is a survey on tropical polytopes from the combinatorial point of view and with a
  focus on algorithms.  Tropical convexity is interesting because it relates a number of
  combinatorial concepts including ordinary convexity, monomial ideals, subdivisions of
  products of simplices, matroid theory, finite metric spaces, and the tropical
  Grassmannians.  The relationship between these topics is explained via one running
  example throughout the whole paper.  The final section explains how the new version
  2.9.4 of the software system \polymake can be used to compute with tropical polytopes.
\end{abstract}

\maketitle

\section{Introduction}

The study of tropical convexity, also known as ``max-plus convexity'', has a long
tradition going back at least to Vorobyev~\cite{Vorbyev67} and
Zimmermann~\cite{Zimmermann77}; see also
\cite{Cunninghame-Green79,LitvinovMaslovShpiz01,CohenGaubertQuadrat05} and the references
there.  Develin and Sturmfels contributed an inherently combinatorial view on the subject,
and they established the link to tropical geometry \cite{DevelinSturmfels04}.  The
subsequent development includes research on how tropical convexity parallels classical
convexity \cite{Joswig05,ArdilaDevelin07,JoswigKulas08,GaubertMeunier08} and linear
algebra \cite{Sergeev07} as well as investigations on the relationship to commutative
algebra \cite{BlockYu06,DevelinYu06}.  Most recently, topics in algebraic and arithmetic
geometry came into focus
\cite{KeelTevelev06,Speyer06,JoswigSturmfelsYu07,HerrmannJensenJoswigSturmfels08}.  The
purpose of this paper is to list known algorithms in this area and to explain how the
connections between the various topics work.  We emphasize the aspects of geometric
combinatorics.

Our version of the \emph{tropical semi-ring} is $(\RR,\min,+)$, and we usually write
``$\oplus$'' instead of ``$\min$'' and ``$\odot$'' instead of ``$+$''.  Of course, it is
just a matter of taste if one prefers ``$\max$'' over ``$\min$''.  Whenever convenient we
will augment the semi-ring with the additively neutral element $\infty$ which is absorbing
with respect to $\odot$.  The \emph{tropical $d$-torus} is the set
$\TA^{d-1}:=\RR^d/\RR\vones_d$, where $\vones_d:=(1,1,\dots,1)$ is the all-ones-vectors of
length~$d$.  Via the maximum-norm on $\RR^d$ and the quotient topology the tropical torus
$\TA^{d-1}$ carries a natural topology which is homeomorphic to $\RR^{d-1}$.
Componentwise tropical addition and tropical scalar multiplication turn $\RR^d$ into a
semi-module, and since tropical scalar multiplication with $\lambda$ is the same as the
ordinary addition of the vector $\lambda\cdot\vones_d$ tropical linear combinations of
elements in $\TA^{d-1}$ are well-defined.  For a set $V\subset\TA^{d-1}$ of
\emph{generators} the \emph{tropical convex hull} is defined as
\[
\tconv V \ := \ \SetOf{\lambda\odot v\oplus\mu\odot w}{\lambda,\mu\in\RR,\, v,w\in V} \, .
\]
If $V$ is finite then $\tconv V$ is a \emph{tropical polytope}, which can also be called a
``min-plus cone''.  There are several natural ways to represent a tropical polytope, for
instance, as the tropical convex hull of points (as in the definition), as the union of
ordinary polytopes \cite{DevelinSturmfels04}, or as the intersection of tropical
halfspaces \cite{Joswig05}.  We will discuss several ``tropical convex hull algorithms'',
that is, algorithms which translate one representation into the other.  This turns out to
be related to algorithms in ordinary convexity as well as to algorithms known from
combinatorial optimization.  Here we focus on the key geometric aspects.

Combinatorics enters the stage through the observation
\cite[Theorem~1]{DevelinSturmfels04} that configurations of $n$ points in $\TA^{d-1}$ or,
equivalently, tropical polytopes with a fixed set of generators, are dual to regular
subdivisions of the product of simplices $\Delta_{n-1}\times\Delta_{d-1}$.  Products of
simplices in turn occur as the vertex figures of hypersimplices, and these hypersimplices
are known to serve as adequate combinatorial models for the Grassmannians.  In fact, it
turns out that the tropical Grassmannians can be approximated in terms of decompositions
of hypersimplices into matroid polytopes
\cite{Kapranov93,SpeyerSturmfels04,Speyer06,HerrmannJensenJoswigSturmfels08}.  We will
show how a configuration of $n$ points in $\TA^{d-1}$ can be lifted to a matroid
decomposition of the hypersimplex $\Delta(d,n+d)$.

The structure of this paper is as follows.  The first section briefly collects some
information about the tropical determinant and tropical hyperplanes before the next one
introduces the basic combinatorial concepts of tropical convexity.  This section also
explains how ordinary convex hull algorithms can be used to compute with tropical
polytopes.  Then we discuss various versions of the tropical convex hull problem.  In a
section on matroid subdivisions we explore the role of tropical polytopes for the tropical
Grassmannians and finite metric spaces.  Finally, we show how the new version 2.9.4 of the
software system \polymake can be used for computing with tropical polytopes.

\section{Tropical Determinants and Tropical Hyperplanes}

Evaluating ordinary determinants is the key primitive operation of many algorithms in
ordinary convexity.  Its tropical counterpart is equally fundamental.  The \emph{tropical
  determinant} of a matrix $M=(m_{ij})\in\RR^{d\times d}$, also called the ``min-plus
permanent'', is defined by the tropicalized Leibniz formula
\begin{equation}\label{eq:tdet}
  \tdet M \ := \ \bigoplus_{\sigma\in\Sym_d} m_{1,\sigma(1)} \odot m_{2,\sigma(2)} \odot
  \dots \odot m_{d,\sigma(d)} \, ,
\end{equation}
where $\Sym_d$ is the symmetric group acting on the set $[d]:=\{1,2,\dots,d\}$.
Evaluating $\tdet$ is the same as solving the \emph{linear assignment problem}, or
``weighted bipartite matching problem'', from combinatorial optimization for the weight
matrix $M$.  This can be done in $O(d^3)$ time \cite[Corollary 17.4b]{Schrijver03}.

By definition $\det M$ \emph{vanishes} if the minimum in the defining
Equation~\eqref{eq:tdet} is attained at least twice. In this case $M$ is \emph{tropically
  singular}, and it is \emph{tropically regular} otherwise.  Checking if the tropical
determinant vanishes can be translated into solving $d+1$ assignment problems as follows:
First we evaluate $\tdet M$ be solving one assignment problem.  This way we find some
optimal permutation $\sigma$ with $\tdet M=m_{1,\sigma(1)} \odot m_{2,\sigma(2)} \odot
\dots \odot m_{d,\sigma(d)}$.  The permutation $\sigma$ is called a \emph{realizer} of
$\tdet M$.  We have to check if there are other realizers or not.  To this end define $d$
matrices $M_1,M_2,\dots,M_d$ all of which differ from $M$ in only one coefficient, namely
the coefficient in the $i$-th row and the $\sigma(i)$-th column of $M_i$ is increased to
$m_{i,\sigma(i)}+1$.  We subsequently compute $\tdet M_1$, $\tdet M_2$, $\ldots$ up to
$\tdet M_d$.  Clearly $\tdet M\le \tdet M_i$ for all $i$.  Now $M$ is tropically singular
if there is another permutation $\tau$ for which the minimum is also attained.  As
$\sigma$ and $\tau$ must differ in at least one place it follows that such a $\tau$ exists
if and only if $\tdet M=\tdet M_i$ for some $i$.  We conclude that $M$ is tropically
regular if and only if $\tdet M<\tdet M_i$ for all~$i$.  Hence deciding if $M$ is
tropically singular or not requires $O(d^4)$ time.

\begin{remark}
  With the tropical determinant we can express that a square matrix in the tropical world
  should be considered as having full rank or not.  Defining the rank of a general matrix
  in the tropical setting is much more subtle \cite{DevelinSantosSturmfels05}.
\end{remark}

The tropical evaluation of a linear form $a\in\RR^d$ at a vector $x\in\TA^{d-1}$ reads
$\langle a,x\rangle_\trop:=a_1\odot x_1\oplus a_2\odot x_2\oplus \dots \oplus a_d\odot
x_d$. Again this expression \emph{vanishes} if the minimum is attained at least twice.
A vector $a\in\RR^d$ defines a \emph{tropical hyperplane}
\[
H(a) \ := \ \SetOf{x\in\TA^d}{\langle a,x\rangle_\trop\text{ vanishes}} \, ,
\]
and the point $-a\in\TA^{d-1}$ is called the \emph{apex} of $H(a)$.  Any two tropical
hyperplanes just differ by a translation.  We have the following tropical analog to the
situation in ordinary linear algebra; see also \cite{DevelinSantosSturmfels05,Izhakian06}.

\begin{proposition}[{\cite[Lemma 5.1]{RichtergebertSturmfelsTheobald05}}]\label{prop:singular}
  The matrix $M$ is tropically singular if and only if its rows (or, equivalently, its
  columns) considered as points in $\TA^{d-1}$ are contained in a tropical hyperplane.
\end{proposition}

The ordering of the real numbers allows to refine the linear algebra in $\RR^{d-1}$ to the
theory of ordinary convexity.  For algorithmic purposes sidedness queries of a point
versus an affine hyperplane (spanned by $d-1$ other points) are crucial, and this is
computed by evaluating the sign of the ordinary determinant of a $d\times d$-matrix with
homogeneous coordinate vectors as its rows.  This does not directly translate to the
tropical situation since the sign of the tropical determinant does not have a geometric
meaning.  More severely, the complement of a tropical hyperplane in $\TA^{d-1}$ has
exactly $d$ connected components, its \emph{open sectors}.  The remedy is the following.
Let us suppose that all the realizers of the tropical determinant of $M$ share the same
parity.  In~\cite{Joswig05} this parity is called the \emph{tropical sign} of $M$, denoted
as $\tsgn M$.  If, however, $M$ has realizers of both signs then $\tsgn M$ is set to zero.
The tropical sign is essentially the same as the sign of the determinant in the
``symmetrized min-plus algebra'' \cite{BacelliCohenOlsderQuadrat92}.  The computation of
the tropical sign of a matrix is equivalent to deciding if a directed graph has a directed
cycle of even length \cite[\S3]{Butkovic95}.  This latter problem is solvable in $O(d^3)$
time \cite{RobertsonSeymourThomas99,McCuaig01}.

For fixed $v_2,v_3,\dots,v_d\in\TA^{d-1}$ the tropical sign gives rise to a map
\[
\tau \,:\, \TA^{d-1}\to\{0,\pm 1\} \,:\, x\mapsto\tsgn(x,v_2,v_3,\dots,v_d) \, ,
\]
where $(x,v_2,v_3,\dots,v_d)$ is the $d\times d$-matrix formed of the given vectors as its
rows. The following is a signed version of Proposition~\ref{prop:singular}.

\begin{theorem}[{\cite[Theorem 4.7 and Corollary 4.8]{Joswig05}}]
  Either $\tau$ is constantly zero, or for $\epsilon=\pm1$ the preimage
  $\tau^{-1}(\{0,\epsilon\})$ is a closed tropical halfspace. Each closed tropical halfspace
  arises in this way.
\end{theorem}

Here a \emph{closed tropical halfspace} in $\TA^{d-1}$ is the union of at least one and at
most $d-1$ closed sectors of a fixed tropical hyperplane.  A \emph{closed sector} is the
topological closure of an open sector. For $x\in\TA^{d-1}$ let
\[
x^0 \ := \ x-\min(x_1,x_2,\dots,x_d)\cdot\vones_d
\]
be the unique representative in the coset $x+\RR\vones_d$ with non-negative coefficients
and at least one zero.  Then the open sectors of the tropical hyperplane $H(0)$ are the
sets $S_1,S_2,\dots,S_d$ where
\[
S_i \ := \ \SetOf{x\in\TA^{d-1}}{x^0_i=0 \text{ and } x^0_j>0 \text{ for $j\ne i$}} \, .
\]
For the closed sectors we thus have $\bar S_i=\smallSetOf{x\in\TA^{d-1}}{x^0_i=0}$.
Clearly, $-a+\bar S_1,-a+\bar S_2,\dots,-a+\bar S_d$ are the closed sector of the tropical
hyperplane $H(a)$ and similarly for the closed sectors.  It was mentioned in the
introduction that $\TA^{d-1}$ is homeomorphic to~$\RR$.  One specific homeomorphism is
given by the map
\[
(x_1,x_2,\dots,x_d)+\RR\vones_d \ \mapsto \ (x_2-x_1,x_3-x_1,\dots,x_d-x_1) \, .
\]
In this paper we will often identify $\TA^{d-1}$ with $\RR^{d-1}$ via this particular map.
For instance, its inverse translates our pictures below to $\TA^2$.  Moreover, it allows
us to discuss matters of ordinary convexity in $\TA^{d-1}$; see \cite{JoswigKulas08}.
Notice that the closed and open sectors of any halfspace are both tropically and
ordinarily convex.  For references to ordinary convexity and, particular, to the theory of
convex polytopes see \cite{Gruenbaum03,Ziegler95}.  In order to avoid confusion which type
of convexity is relevant in a particular statement we will explicity say either
``tropical'' or ``ordinary'' throughout.

\section{Vertices, Pseudo-vertices, and Types}

The combinatorics sneaks into the picture through the choice of a system of generators of
a tropical polytope.  Let $V=(v_1,v_2,\dots,v_n)$ be a finite ordered sequence of points
in~$\TA^{d-1}$.  This induces a cell decomposition of $\TA^{d-1}$ into \emph{types} as
follows
\[
\type_V(x) \ := \ (T_1,T_2,\dots,T_d) \, ,
\]
where $T_k=\smallSetOf{i\in[n]}{v_i\in x+\bar S_k}$.  The individual sets $T_k$ are called
\emph{type entries}.  A type vector $(T_1,T_2,\dots,T_d)$ with respect to $V$ satisfies
the condition $T_1\cup T_2\cup\dots\cup T_d=[n]$ (but the converse does not hold).
Sometimes it will be convenient to identify the point sequence $V$ in $\RR^d$ with the
matrix $(v_{ij})$ whose $i$-th row is the point $v_i$.

\begin{proposition}[{\cite[Lemma~10]{DevelinSturmfels04}}]
  \label{prop:DS:Lemma10}
  The points of a fixed type $T$ with respect to the generating system $V$ form an
  ordinary polyhedron $X_T$ which is tropically convex.  More precisely,
  \[
  X_T \ = \ \bigSetOf{x\in\TA^{d-1}}{x_k-x_j\le v_{ik}-v_{ij} \text{ for all } j,k\in[d]
    \text{ and } i\in T_j} \, .
  \]
\end{proposition}

A set in $\TA^{d-1}$ is \emph{tropically convex} if it coincides with its own tropical
convex hull.  The ordinary polyhedron $X_T$ is bounded if and only if all entries in the
type $T$ are non-empty.  The bounded ordinary polyhedra $X_T$ are precisely the
\emph{polytropes} studied in \cite{JoswigKulas08}.

\begin{theorem}[{\cite[Theorem~15]{DevelinSturmfels04}}]\label{thm:type-decomposition}
  The collection of ordinary polyhedra $X_T$, where $T$ ranges over all types, is a
  polyhedral decomposition of $\TA^{d-1}$.  The tropical polytope $\tconv V$ is precisely
  the union of the bounded types.
\end{theorem}

Let $\cU$ be the $(n+d-1)$-dimensional real vector space
$\RR^{n+d}/\RR(\vones_n,-\vones_d)$; its elements are the equivalence classes of
pairs $(y,z)+\RR(\vones_n,-\vones_d)$ where $y\in\RR^n$ and $z\in\RR^d$.  For an $n\times
d$-matrix $V=(v_{ij})$ the polyhedron
\begin{equation}\label{eq:P_V}
\cP_V \ := \ \bigSetOf{(y,z)\in\cU}{y_i+z_j\le v_{ij} \text{ for all } i\in [n] \text{ and }
  j\in [d]}
\end{equation}
is unbounded.  We call the polyhedron $\cP_V$ the \emph{envelope} of $\tconv V$ with
respect to $V$.  By \cite[Lemma~22]{DevelinSturmfels04} the map
$\cU\rightarrow\TA^{d-1}:(y,z)\mapsto z$ sends the bounded faces of $\cP_V$ to the bounded
types induced by the generating set~$V$.  In particular, $\tconv V$ is the image of the
union of all bounded faces of its envelope.  Similarly the unbounded types correspond to
the unbounded faces of $\cP_V$.  Of particular importance to us are the vertices of
$\cP_V$ which we call the \emph{pseudo-vertices} of $\tconv V$ with respect to the
generating system~$V$.  What we just discussed is explicitly stated as
Algorithm~\ref{algo:pseudo-vertices:ordinary} below.

\begin{algorithm}
  \dontprintsemicolon

  \Input{$V\subset\TA^{d-1}$ finite}
  \Output{pseudo-vertices of $\tconv V$}

  compute the vertices $W$ of the envelope $\cP_V$ via an arbitrary (dual) ordinary convex
  hull algorithm \;
  
  \Return{image of $W$ under projection $(y,z)\mapsto z$} \;

  \caption{Computing the pseudo-vertices via ordinary convex hull.}
  \label{algo:pseudo-vertices:ordinary}
\end{algorithm}

The ordinary convex hull problem asks to compute the facets of the ordinary convex hull of
a given finite set of points in Euclidean space.  Via cone polarity this is equivalent to
the dual problem of enumerating the ordinary vertices from an ordinary halfspace
description.  The latter also extends without changes to ordinary unbounded polyhedra
which do not contain any affine subspace, and this is why a dual convex hull algorithm can
be applied to $\cP_V$ as given in \eqref{eq:P_V} in
Algorithm~\ref{algo:pseudo-vertices:ordinary}.  For an overview of ordinary convex hull
algorithms both from the theoretical and the practical point of view see
\cite{AvisBremnerSeidel97,Beneath-Beyond}.  The complexity of the ordinary convex hull
problem in variable dimension is not entirely settled.

\begin{remark}
  It is obvious to compute the types of the pseudo-vertices with respect to a given set of
  generators by checking the definition.  More interesting is that it also works the other
  way around.  Once the type of a pseudo-vertex is given one can determine its coordinates
  from the coordinates of the generators.  To see this observe that if $w$ is a
  pseudo-vertex with respect to $V$ then after fixing an arbitrary coordinate, say $w_1$,
  the inequalities in Proposition~\ref{prop:DS:Lemma10} degenerate to a system of
  equations which determine $w_2-w_1,w_3-w_1,\dots,w_d-w_1$.
\end{remark}

For $\tconv V$, just considered as a subset of $\TA^{d-1}$, there is a unique minimal
system of generators with respect to inclusion; see \cite[Theorem
15.6]{Cunninghame-Green79} and \cite[Proposition 21]{DevelinSturmfels04}.  These are the
\emph{tropical vertices} of $\tconv V$.  Among the generators a vertex $w$ is recognized
by the property that at least one of its type entries contains only the index of $w$
itself; see Algorithm~\ref{algo:vertices}.  This process takes $O(nd^2)$ time.

\begin{algorithm}
  \dontprintsemicolon

  \Input{$V\subset\TA^{d-1}$ finite}
  \Output{set of tropical vertices of $\tconv V$}

  $W\leftarrow\emptyset$ \;
  \ForEach{$w\in V$}{
    $(T_1,T_2,\dots,T_d)\leftarrow\type_{V\setminus\{w\}}(w)$\;
    \If{$T_k=\emptyset$ for some $k$}{
      add $w$ to $W$ \;
    }
  }
  \Return{W} \;

  \caption{Computing the tropical vertices.}
  \label{algo:vertices}
\end{algorithm}

\begin{figure}[htb]
  \includegraphics[height=.45\textwidth]{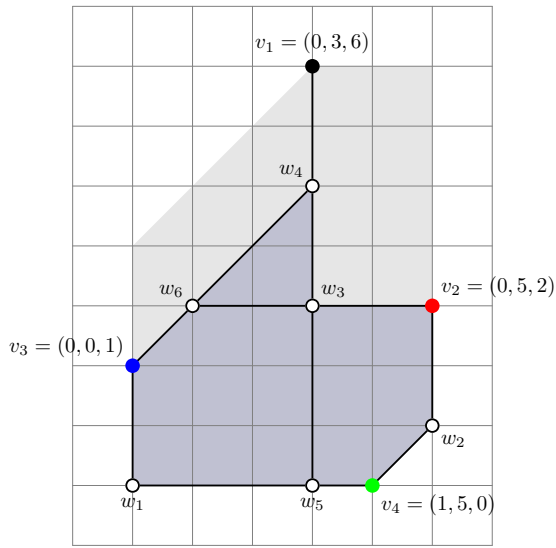}

  \caption{Tropical convex hull of four points in $\TA^2$.}
  \label{fig:BlockYu-Fig1}
\end{figure}

\begin{example}\label{ex:BlockYu-Fig1:1}
  Consider the point sequence $V$ with
  \[
  v_1 = (0,3,6), \; v_2 =(0,5,2), \; v_3=(0,0,1), \; v_4=(1,5,0)
  \]
  in $\TA^2$.  This is the same example as the one considered in
  \cite[Figure~1]{BlockYu06}.  It is shown in Figure~\ref{fig:BlockYu-Fig1}.  The tropical
  polygon $\tconv V$ has precisely ten pseudo-vertices. They have their coordinates and
  types as listed in Table~\ref{tab:BlockYu-Fig1:types}.  The first four pseudo-vertices
  are the given generators, and these are also the tropical vertices.  The meaning of the
  last column will be explained in Section~\ref{sec:halfspaces} below.  Altogether the
  type decomposition of $\tconv V$ with respect to $V$ has ten vertices, 12 edges, and
  three two-dimensional faces.  This information is collected in the so-called
  \emph{$f$-vector} $(10,12,3)$.
\end{example}

\begin{table}[hbt]
  \caption{Pseudo-vertices of the tropical polygon from Example~\ref{ex:BlockYu-Fig1:1}
    and Figure~\ref{fig:BlockYu-Fig1}.}
  \renewcommand{\arraystretch}{0.9}
  \begin{tabular*}{.75\linewidth}{@{\extracolsep{\fill}}clcc@{}}
    \toprule
    label & \multicolumn{1}{c}{$w$} & $\type_V(w)$  & facet sectors\\
    \midrule
    $v_1$ & $(0,3,6)$ & {1}, {1}, {1234}  & 3\\
    $v_2$ & $(0,5,2)$ & {2}, {123}, {24}  & \\
    $v_3$ & $(0,0,1)$ & {123}, {3}, {34}  & \\
    $v_4$ & $(1,5,0)$ & {24}, {134}, {4}  & \\[1ex]
    $w_1$ & $(1,1,0)$ & {1234}, {3}, {4}  & 1\\
    $w_2$ & $(0,5,0)$ & {2}, {1234}, {4}  & 2\\
    $w_3$ & $(0,3,2)$ & {12}, {13}, {24}  & 23\\
    $w_4$ & $(0,3,4)$ & {1}, {13}, {234}  & 13\\
    $w_5$ & $(1,4,0)$ & {124}, {13}, {4}  & \\
    $w_6$ & $(0,1,2)$ & {12}, {3}, {234}  & \\
    \bottomrule
  \end{tabular*}
  \label{tab:BlockYu-Fig1:types}
\end{table}

Let $\Delta_k:=\conv\{e_1,e_2,\dots,e_{k+1}\}$ be the (ordinary) regular $k$-dimensional
simplex.  Here and in the sequel the standard basis vectors of $\RR^d$ are denoted as
$e_1,e_2,\dots,e_d$.  The product of two simplices $\Delta_{n-1}\times\Delta_{d-1}$ is an
$(n+d-2)$-dimensional ordinary polytope with $n+d$ facets and $nd$ vertices.  We can read
our matrix $V$ as a way of assigning the height $v_{ij}$ to the vertex $(e_i,e_j)$ of
$\Delta_{n-1}\times\Delta_{d-1}$. Projecting back the lower convex hull of the lifted
polytope yields a regular subdivision of $\Delta_{n-1}\times\Delta_{d-1}$; see
\cite{Triangulations} for a comprehensive treatment of the subject.  The following is the
main structural result about tropical polytopes.

\begin{theorem}[{\cite[Theorem~1]{DevelinSturmfels04}}]\label{thm:main}
  The tropical polytope $\tconv V$ with the induced structure as a polytopal complex by
  its bounded types is dual to the regular subdivision of $\Delta_{n-1}\times\Delta_{d-1}$
  induced by $V$.
\end{theorem}

In view of this result the most natural version of the tropical convex hull problem
perhaps is the one that asks to compute the pseudo-vertices together with all the bounded
types.  This can be achieved via applying a suitable modification of an algorithm of
Kaibel and Pfetsch \cite{KaibelPfetsch02} for computing the face lattice of an ordinary
polytope from its vertex facet incidences to the output of
Algorithm~\ref{algo:pseudo-vertices:ordinary}.

For ordinary convex hull computations it is known that genericity assumptions allow for
additional and sometimes faster algorithms.  This is also the case in the tropical situation.
The point set $V\subset\TA^{d-1}$ is \emph{sufficiently generic} if no $k\times
k$-submatrix of $V$, viewed as a matrix, is tropically singular.  This condition is
equivalent to the property that the envelope $\cP_V$ is a simple polyhedron.  This is the
case if and only if the induced subdivision of the product of simplices is a
triangulation.  A $k$-dimensional polyhedron is \emph{simple} if each vertex is contained
in exactly $k$ facets; for details about polytopes and polyhedra see \cite{Ziegler95}.
The maximal cells of the polytopal complex dual to the type decomposition of $\tconv V$
are in bijection with the pseudo-vertices. If $(T_1,T_2,\dots,T_d)$ is the type of the
pseudo-vertex $w$ then the corresponding maximal cell contains precisely those vertices
$(e_i,e_j)$ of $\Delta_{n-1}\times\Delta_{d-1}$ which satisfy $i\in T_j$.

\begin{example}
  In fact, the point set $V$ in Example~\ref{ex:BlockYu-Fig1:1} is sufficiently generic,
  and hence the subdivision of $\Delta_3\times\Delta_2$ induced by $V$ is a triangulation.
  For instance, the tropical vertex $v_1$ has the type $({1}, {1}, {1234})$, and it
  corresponds to the maximal cell with vertices
  \[ \{(e_1,e_1),(e_1,e_2),(e_1,e_3),(e_2,e_3),(e_3,e_3),(e_4,e_3)\} \, ,\] a $5$-simplex.
  See the Table~\ref{tab:triangulation} for the complete list of pseudo-vertices versus
  maximal cells in the triangulation; the vertex $(e_i,e_j)$ is abbreviated as $ij$.
\end{example}

\begin{table}[hbt]
  \caption{Triangulation of $\Delta_3\times\Delta_2$ dual to $\tconv V$ decomposed into types.}
  \label{tab:triangulation}
  \renewcommand{\arraystretch}{0.9}
  \begin{tabular*}{.75\linewidth}{@{\extracolsep{\fill}}ccc@{}}
    \toprule
    label & maximal cell in triangulation & min.\ generator of $I^*$ \\
    \midrule
    $v_1$ & $11$ $12$ $13$ $23$ $33$ $43$ & $x_{21}x_{22}x_{31}x_{32}x_{41}x_{42}$ \\ 
    $v_2$ & $12$ $21$ $22$ $23$ $32$ $43$ & $x_{11}x_{13}x_{31}x_{33}x_{41}x_{42}$ \\ 
    $v_3$ & $11$ $21$ $31$ $32$ $33$ $43$ & $x_{12}x_{13}x_{22}x_{23}x_{41}x_{42}$ \\ 
    $v_4$ & $12$ $21$ $32$ $41$ $42$ $43$ & $x_{11}x_{13}x_{22}x_{23}x_{31}x_{33}$ \\[1ex] 
    $w_1$ & $11$ $21$ $31$ $32$ $41$ $43$ & $x_{12}x_{13}x_{22}x_{23}x_{33}x_{42}$ \\ 
    $w_2$ & $12$ $21$ $22$ $32$ $42$ $43$ & $x_{11}x_{13}x_{23}x_{31}x_{33}x_{41}$ \\ 
    $w_3$ & $11$ $12$ $21$ $23$ $32$ $43$ & $x_{13}x_{22}x_{31}x_{33}x_{41}x_{42}$ \\ 
    $w_4$ & $11$ $12$ $23$ $32$ $33$ $43$ & $x_{13}x_{21}x_{22}x_{31}x_{41}x_{42}$ \\ 
    $w_5$ & $11$ $12$ $21$ $32$ $41$ $43$ & $x_{13}x_{22}x_{23}x_{31}x_{33}x_{42}$ \\ 
    $w_6$ & $11$ $21$ $23$ $32$ $33$ $43$ & $x_{12}x_{13}x_{22}x_{31}x_{41}x_{42}$ \\ 
    \bottomrule
  \end{tabular*}
\end{table}

Block and Yu \cite{BlockYu06} investigated the relationship of tropical convexity to
commutative algebra.  We will review some of their ideas in the following.  The situation
is particularly clear if the points considered are sufficiently generic.

Let $\KK[x_{11},x_{12},\dots,x_{nd}]$ be the polynomial ring in $nd$ indeterminates over
the field $\KK$. The indeterminate $x_{ij}$ will be assigned the \emph{weight} $v_{ij}$,
the $j$-th coordinate of the point $v_i$ in our sequence $V$.  \emph{Weights} of monomials
are extended additively, that is, the weight of $x^\alpha=\prod x_{ij}^{\alpha_{ij}}$ is
$\sum\alpha_{ij}v_{ij}$.  The \emph{weight} $\initial_V(f)$ of polynomial $f$ is the sum
of its terms of maximal weight.  This defines a partial term ordering on
$\KK[x_{11},x_{12},\dots,x_{nd}]$.  Let $J$ be the determinantal ideal generated by all
$2\times 2$-minors.  Then
\[
\initial_V(J) \ := \ \GeneratedBy{\initial_V(f)}{f\in J}
\]
is the \emph{initial ideal} of $J$ with respect to $V$.

\begin{proposition}[{\cite[Proposition~4]{BlockYu06}}]
  The points in $V$ are sufficiently generic if and only if $\initial_V(J)$ is a
  square-free monomial ideal.
\end{proposition}

A square-free monomial in $\KK[x_{11},x_{12},\dots,x_{nd}]$ corresponds to a subset of the
indeterminates $x_{11},x_{12},\dots,x_{nd}$ and conversely.  The unique minimal set of
generators of a square-free monomial ideal is the set of minimal non-faces of a unique
finite simplicial complex on the vertices $x_{11},x_{12},\dots,x_{nd}$.  Starting with our
ideal $I:=\initial_V(J)$ this simplicial complex is the \emph{initial complex}
$\Delta_V(J)$.  The reverse construction can be applied to any finite simplicial complex:
Its minimal non-faces generate the \emph{Stanley-Reisner} ideal of the complex.  In
particular, the initial ideal $I$ is the Stanley-Reisner ideal of the initial complex
$\Delta_V(J)$.  The complements of the maximal faces of $\Delta_V(J)$ generate another
squarefree monomial ideal, the \emph{Alexander dual} $I^*$.  For details see
\cite[Chapter~1]{MillerSturmfels05}.  The whole point of this discussion is that
$\Delta_V(J)$ coincides with the triangulation of $\Delta_{n-1}\times\Delta_{d-1}$ induced
by~$V$ \cite[Lemma~5]{BlockYu06}.  We arrive at
Algorithm~\ref{algo:pseudo-vertices:duality} below as an alternative to
Algorithm~\ref{algo:pseudo-vertices:ordinary}.

\begin{algorithm}
  \dontprintsemicolon

  \Input{$V\subset\TA^{d-1}$ finite, in general position}
  \Output{types of the pseudo-vertices of $\tconv V$}

  $J \leftarrow \langle \text{$2\times 2$-minors of the $n\times d$-matrix $(x_{ij})$}\rangle$ \;
  $I \leftarrow \initial_V(J)$ \;
  $I^* \leftarrow \text{Alexander dual of $I$}$ \;
  \Return $I^*$ \;

  \caption{Computing the pseudo-vertices via Alexander duality.}
  \label{algo:pseudo-vertices:duality}
\end{algorithm}

\begin{example}\label{ex:BlockYu-Fig1:2}
  We continue our Example~\ref{ex:BlockYu-Fig1:1}.  In this case the initial ideal
  $I:=\initial_V(J)$ reads
  \begin{align*}
    I \ = \ \langle
    &x_{11}x_{22},
    x_{11}x_{42},
    x_{12}x_{21}x_{33},
    x_{12}x_{31},
    x_{13}x_{21},
    x_{13}x_{22},
    x_{13}x_{31},
    x_{13}x_{32},
    x_{13}x_{41},
    x_{13}x_{42},\\
    &x_{22}x_{31},
    x_{22}x_{33},
    x_{22}x_{41},
    x_{23}x_{31},
    x_{23}x_{41},
    x_{23}x_{42},
    x_{31}x_{42},
    x_{33}x_{41},
    x_{33}x_{42}
    \rangle \, .
  \end{align*}
  Notice that the generator $x_{12}x_{21}x_{33}$ is of degree three, whence $I$ is not
  homogeneous.  The Alexander dual $I^*$ is generated by the monomials listed in the third
  column of Table~\ref{tab:BlockYu-Fig1:types}.
  
  We would like to point out that there are typos in the description of $I$ in \cite[page
  109]{BlockYu06}, but the description of $I^*$ is correct.  The initial complex
  $\Delta_V(J)$, or equivalently the triangulation dual to $\tconv V$ subdivided into its
  bounded types, is the pure $5$-dimensional simplicial complex whose facets are the
  complements of the sets corresponding to the ten minimal generators of $I^*$; the
  maximal cells of $\Delta_V(J)$ are listed in the second column of
  Table~\ref{tab:BlockYu-Fig1:types}.  It is known that the $f$-vector of any two
  triangulations of a product of simplices (without new vertices) is the same; in this
  case it reads $(12, 48, 92, 93, 48, 10)$.
\end{example}

The following is the main algebraic result on this topic.

\begin{theorem}[{\cite[Theorem 1]{BlockYu06}}]\label{thm:resolution}
  If the points in $V$ are sufficiently generic then $\tconv V$ supports a minimal free
  resolution of the ideal $I^*$, as a polytopal complex.
\end{theorem}

This says that the combinatorics of the type decomposition of $\tconv V$ completely
controls the relations among the generators of the ideal $I^*$.  For details on cellular
resolutions the reader is referred to \cite[Chapter~4]{MillerSturmfels05}.

\begin{example}
  We explain in slightly more detail where this leads to for our running example.
  Abbreviating $S:=\KK[x_{11},x_{12},\dots,x_{43}]$, the cellular resolution of
  Theorem~\ref{thm:resolution} is the \emph{complex} (in the homological sense)
  \begin{equation}\label{eq:resolution}
    0 \ \leftarrow S^1 \ \leftarrow \ S^{10} \ \leftarrow \ S^{12} \ \leftarrow \ S^3 \ \leftarrow 0
  \end{equation}
  of free $S$-modules.  The degrees $10$, $12$, and $3$ of the modules correspond to the
  $f$-vector of the type decomposition of $\tconv V$; see Example~\ref{ex:BlockYu-Fig1:1}.
  The arrows are homomorphisms of $S$-modules (to be read off the combinatorics of $\tconv
  V$) with the property that concatenating two consecutive arrows gives the zero map; the
  precise maps for this specific example are given in \cite[Example~10]{BlockYu06}.
  Knowing the complex~\eqref{eq:resolution} (with its maps) allows to reconstruct the
  combinatorics of the types induced by $V$.
\end{example}

Even if one is not interested in algebraic applications, Theorem~\ref{thm:resolution}
opens up an additional line of attack for the algorithmic problem to compute all the types
together with the pseudo-vertices.  This is due to the fact that a simple polyhedron, such
as the envelope $\cP_V$ in the sufficiently generic case, supports a unique minimal free
resolution on the polytopal subcomplex of its bounded faces \cite[Remark~7]{BlockYu06}.
That is to say, in order to compute the types it suffices to compute any free resolution
and to simplify it to a minimal one afterwards \cite[Algorithm~2]{BlockYu06}.

\begin{remark}
  In ordinary convexity there are at least two algorithms to compute convex hulls
  essentially via sidedness queries: the beneath-and-beyond method (which iteratively
  produces a triangulation) and gift wrapping \cite{AvisBremnerSeidel97,Beneath-Beyond}.
  In general, it is not clear to what extent these also exist in tropical versions.
  However, in the planar case, that is, $d=3$, the cyclic ordering of the tropical
  vertices can be computed by an output-sensitive algorithm in $O(n\log h)$ time
  \cite[Theorem~5.3]{Joswig05}, where $n$ is the number of generators and $h$ is the
  number of tropical vertices, by a suitable translation from the situation in the
  ordinary case \cite{Chan96}.
\end{remark}

\section{Minimal Tropical Halfspaces}
\label{sec:halfspaces}

A key theorem about ordinary polytopes states that an ordinary convex polytope, that is,
the ordinary convex hull of finitely many points in $\RR^d$ is the same as the
intersection of finitely many affine halfspaces \cite[Theorem 2.15]{Ziegler95}. For
tropical polytopes there is a similar statement.

\begin{theorem}[{\cite[Theorem~3.6]{Joswig05}}]
  The tropical polytopes are precisely the bounded intersections of tropical halfspaces.
\end{theorem}

For $a\in\TA^{d-1}$ and $H\subset[d]$ we use the notation $(a,H)$ for the closed tropical
halfspace $a+\bigcup_{i\in H} \bar S_i$.  A tropical halfspace is \emph{minimal} with
respect to a tropical polytope $\tconv V$ if it is minimal with respect to inclusion among
all tropical halfspaces containing $\tconv V$.  This raises the question how these can be
computed.  We begin with a straightforward procedure to check if one tropical halfspace is
contained in another.

\begin{lemma}\label{lem:halfspace-inclusion}
  Let $(a,H)$ and $(b,K)$ be tropical halfspaces.  Then $(a,H)\subseteq(b,K)$ if and only
  if $H\subseteq K$ and $a-b\in\bar S_k$.
\end{lemma}

\begin{proof}
  The property $(a,H)\subseteq(b,K)$ is equivalent to the following: for all $x\in (0,H)$
  we have $a+x-b \in (0,K)$.  So assume the latter.  Then, as $x\in (0,H)$, this implies
  $a-b\in (0,K)$.  Now we can assume that $a=b$ without loss of generality and $H\subseteq
  K$ is immediate.  The converse follows in a similar way.
\end{proof}

It is a feature of tropical convexity that $d$ of the minimal tropical halfspaces of a
tropical polytope can be read off the generator matrix $V$ right away.  To this end we
define the \emph{$k$-th corner} of $\tconv V$ as
\[
c_k(V) \ := \ (-v_{1,k})\odot v_1 \oplus (-v_{2,k})\odot v_2 \oplus \dots \oplus (-v_{n,k})\odot v_n \, .
\]
By construction it is clear that each corner is a point in $\tconv V$.

\begin{lemma}\label{lem:corner}
  The closed tropical halfspace $c_k(V)+\bar S_k$ is minimal with respect to $\tconv V$.
\end{lemma}

\begin{proof}
  By symmetry we can assume that $k=1$.  Then the vectors $(-v_{1,1})\odot v_1,\,
  (-v_{2,1})\odot v_2, \dots ,(-v_{n,1})\odot v_n$ are the Euclidean coordinate vectors of
  the points $v_1,v_2,\dots,v_n$, and their pointwise minimum $c:=c_1(V)$ is the ``lower
  left'' corner of their tropical convex hull.  By construction the tropical halfspace
  $c+\bar S_1$ contains $\tconv V$.  Suppose that $c+\bar S_1$ is not minimal.  Then there
  must be some other tropical halfspace $w+\bar S_K$ contained in $c+\bar S_1$ which still
  contains $\tconv V$.  By Lemma~\ref{lem:halfspace-inclusion} we have $K=\{1\}$.
  Moreover, since $c\in\tconv V$, we have $c-w\in\bar S_1$ and so $c+\bar S_1\subseteq
  w+\bar S_1$, again by Lemma~\ref{lem:halfspace-inclusion}.  We conclude that $w=c$, and
  this completes our proof.
\end{proof}

It is now a consequence of the following proposition that the corners are pseudo-vertices.

\begin{proposition}[{\cite[Proposition 3.3]{Joswig05}}]
  The apex of a minimal tropical halfspace with respect to $\tconv V$ is a pseudo-vertex
  with respect to $\tconv V$.
\end{proposition}

The tropical halfspace $c_k(V)+\bar S_k$ is called the \emph{$k$-th cornered} tropical
halfspace with respect to $\tconv V$.  The \emph{cornered hull} of $\tconv V$ is the
intersection of all $d$ cornered tropical halfspaces.  Each tropical halfspace is
tropically convex.  But a tropical halfspace is convex in the ordinary sense if and only
if it consists of a single sector.  This implies that the cornered hull of a tropical
polytope is tropically convex and also convex in the ordinary sense.  Since it is also
bounded it is a polytrope.  In fact, the cornered hull of $V$ is the smallest polytrope
containing $V$.

\begin{example}
  For the matrix $V$ introduced in Example \ref{ex:BlockYu-Fig1:1} the three corners are
  the pseudo-vertices
  \[
  c_1(V)=w_1=(1,1,0), \quad c_2(V)=w_6=(0,5,0), \quad \text{and} \quad c_3(V)=v_1=(0,3,6) \, .
  \]
  This notation fits with Table~\ref{tab:BlockYu-Fig1:types}.  The cornered hull is shaded
  lightly in Figure~\ref{fig:BlockYu-Fig1}.
\end{example}

\begin{lemma}
  Let $(a,H)$ be any tropical halfspace containing $V$, and let
  $(T_1,T_2,\dots,T_d)=\type_V(a)$.  Then $\bigcup_{j\in H}T_j=[n]$.
\end{lemma}

\begin{proof}
  The type entry $T_j$ contains (the indices of) the generators in $V$ which are contained
  in the closed sector $a+\bar S_j$.  Since all the generators are contained in each facet
  defining tropical halfspace we have $\bigcup_{j\in H}T_j=[n]$.
\end{proof}

From this lemma it is clear that for a minimal tropical halfspace containing $\tconv V$
the type entries of $a$ with respect to $V$ form a set covering of $[n]$ which is minimal
with respect to inclusion.  It may happen that one pseudo-vertex occurs as the apex of two
distinct minimal tropical halfspaces.  A tropical halfspace is \emph{locally minimal} if
it is minimal with respect to inclusion among all tropical halfspaces with a fixed apex
and containing $V$.

\begin{algorithm}[htb]
  \dontprintsemicolon

  \Input{$x\in\TA^d$ point in the boundary of a tropical polytope $\tconv V$}
  \Output{set of tropical halfspaces with apex $x$ locally minimal with respect to $P$}

  compute $T=(T_0,\dots,T_d)=\type_V(x)$ \;
  compute all (with respect to inclusion) minimal set coverings of $[n]$ by $T_0,\dots,T_d$ \;
  $\cH\leftarrow\text{pairs of $x$ with all these minimal set covers}$ \;
  \Return{$\cH$} \;

  \caption{Computing the minimal tropical halfspaces (locally).}
  \label{algo:local-facet-defining}
\end{algorithm}

The input to Algorithm~\ref{algo:global-facet-defining} below can be the set of
pseudo-vertices with respect to any generating system, in particular, with respect to the
tropical vertices.

\begin{algorithm}[hbt]
  \dontprintsemicolon

  \Input{$W\subset\TA^d$ pseudo-vertices of $\tconv V$}
  \Output{set of minimal tropical halfspaces with respect to $\tconv V$}

  $\cH\leftarrow\emptyset$ \;
  \ForEach{$x\in W$}{
    $\cH'\leftarrow\text{locally minimal tropical halfspaces at $x$ according to
      Algorithm~\ref{algo:local-facet-defining}}$ \;
    \ForEach{$(x,H')\in\cH'$}{
      \ForEach{$(y,H)\in\cH$}{
        \eIf{$(x,H')$ is contained in $(y,H)$}{
          remove $(y,H)$ from $\cH$ \;
        }{
          \If{$(y,H)$ is contained in $(x,H')$}{
            remove $(x,H')$ from $\cH'$ \;
          }
        }
      }
    }
    $\cH\leftarrow\cH\cup\cH'$ \;
  }
  \Return{$\cH$} \;

  \caption{Computing the minimal tropical halfspaces (globally).}
  \label{algo:global-facet-defining}
\end{algorithm}

Notice that, as stated, Algorithm~\ref{algo:local-facet-defining} contains the NP-complete
problem of finding a set covering of $[n]$ of minimal cardinality as a subproblem.  The
following example raises the question if, in order to compute all (globally) minimal
tropical halfspaces it can be avoided to compute \emph{all} locally minimal tropical
halfspaces for all pseudo-vertices.  This is not clear to the author.

\begin{example}\label{ex:BlockYu-Fig1:3}
  Consider the tropical convex hull $\tconv(v_1,v_2,v_3,v_4)$ shown in
  Figure~\ref{fig:BlockYu-Fig1}.  The points $w_4=(0,3,2)$ and $w_5=(0,3,4)$ both are
  pseudo-vertices with respect to the given system of generators, see
  Table~\ref{tab:BlockYu-Fig1:types}.  At $w_5$ there are precisely two locally minimal
  tropical halfspaces, namely $(w_5,\{1,3\})$ and $(w_5,\{2,3\})$.  The only locally
  minimal tropical halfspace at $w_4$ is $(w_4,\{2,3\})$.  As $(w_4,\{2,3\})$ is contained
  in $(w_5,\{2,3\})$, the latter cannot be (globally) minimal.
  
  Both $(w_4,\{2,3\})$ and $(w_5,\{1,3\})$ are (globally) minimal.  Altogether there are
  five minimal tropical halfspaces: the three remaining ones are $(v_1,\{3\})$,
  $(w_1,\{1\})$, and $w_6,\{2\})$.  For each pseudo-vertex which is the apex of a minimal
  tropical halfspace the corresponding \emph{facet sectors} are listed in the final column
  of Table~\ref{tab:BlockYu-Fig1:types}.
\end{example}

\section{Matroid Subdivisions and Tree Arrangements}

It turns out that tropical polytopes or, dually, regular subdivisions of products of
ordinary simplices carry information which can be, in a way, extended to matroid
subdivisions of hypersimplices.  These are known to be relevant for studying questions on
the Grassmannians of $d$-planes in $n$-space over some field $\KK$
\cite{Kapranov93,SpeyerSturmfels04,HerrmannJensenJoswigSturmfels08}.

An ordinary polytope in $\RR^m$ whose vertices are $0/1$-vectors is a \emph{matroid
  polytope} if each edge is parallel to a difference of basis vectors $e_i-e_j$ for some
distinct $i$ and $j$ in $[m]$.  A \emph{matroid} of rank $r$ on the set $[m]$ is a set of
$r$-element subsets of $[m]$ whose characteristic vectors are the vertices of a matroid
polytope.  The elements of a matroid are its \emph{bases}.  It is a result of Gel\cprime
fand, Goresky, MacPherson, and Serganova \cite{GelfandGoreskyMacPhersonSerganova87} that
this definition describes the same kind of objects as more standard ones~\cite{White87}.
Each $m\times n$-matrix $M$ (with coefficients in an arbitrary field $\KK$) gives rise to
a matroid $\cM$ of rank $r$ on the set $[m]$ as follows: the columns of $M$ are indexed by
$[m]$, and $r$ is the rank of $M$.  Those $0/1$-vectors which correspond to bases of the
column space of $M$ are the bases of $\cM$.  We denote the matroid polytope of a matroid
$\cM$ as $\MatroidPolytope{\cM}$.  The \emph{$(r,m)$-hypersimplex} $\Delta(r,m)$ is the
matroid polytope of the \emph{uniform matroid} of rank $r$ on the set $[m]$, that is, the
matroid whose set of bases is $\tbinom{[m]}{r}$, by which we denote the set of all
$r$-element subsets of $[m]$.

The relationship to tropical convexity works as follows.  By the Separation Theorem from
ordinary convexity \cite[\S2.2]{Gruenbaum03} each vertex $v$ of an ordinary polytope $P$
can be separated from the other vertices by an affine hyperplane $H$.  The intersection of
$P$ with $H$ is again an ordinary polytope, and its combinatorial type does not depend on
$H$; this is the \emph{vertex figure} $P/v$ of $P$ with respect to $v$.  Moreover,
whenever we have a lifting function $\lambda$ on $P$ inducing a regular subdivision, then
$\lambda$ induces a lifting function of each vertex figure $P/v$, and this way we obtain a
regular subdivision of $P/v$.  In the case of the $(d,n+d)$-hypersimplex each vertex
figure is isomorphic to the product of simplices $\Delta_{d-1}\times\Delta_{n-1}$.  Hence
regular subdivisions of $\Delta(d,n+d)$ yield configurations of $n$ points in $\TA^{d-1}$
at each vertex of $\Delta(d,n+d)$.  The situation for the hypersimplices is special in
that the converse also holds, that is, each regular subdivision of a vertex figure can be
lifted to a (particularly interesting kind of) regular subdivision of $\Delta(d,n+d)$.

\begin{proposition}[{\cite[Corollary 1.4.14]{Kapranov93}}]
  Each configuration of $n$ points in $\TA^{d-1}$ can be lifted to a regular matroid
  decomposition of $\Delta(d,n+d)$.
\end{proposition}

A \emph{matroid subdivision} of a matroid polytope is a polytopal subdivision with the
property that each cell is a matroid polytope.  Kapranov's original proof
\cite{Kapranov93} uses non-trivial methods from algebraic geometry.  Here we give an
elementary proof, which makes use of the techniques developed in
\cite{HerrmannJensenJoswigSturmfels08}.

\begin{proof}
  Let $V$ be the $d\times n$-matrix which has the $n$ given points as its columns, and let
  $\bar V$ be the $d\times n$-matrix produced by concatenating $V$ with the tropical
  identity matrix of size $d\times d$.  The tropical identity matrix has zeros on the
  diagonal and infinity off the diagonal.  Each $d$-subset $S$ of $[n+d]$ defines $d$
  columns of $\bar V$ and this way a $d\times d$-submatrix $\bar V_S$.  It is easy to
  see that the map
  \[
  \pi \, : \, S \, \mapsto \, \tdet \bar V_S
  \]
  is a finite tropical Pl\"ucker vector.  A \emph{finite tropical Pl\"ucker vector} is a
  map from the set $\tbinom{[n+d]}{d}$ to the reals satisfying the following property: For
  each subset $T$ of $[n+d]$ with $d-2$ elements the minimum
  \[
  \min \bigl\{ \pi(Tij) + \pi(Tkl), \, \pi(Tik) + \pi(Tjl), \, \pi(Til) + \pi(Tjk) \bigr\}
  \]
  is attained at least twice, where $i,j,k,l$ are the pairwise distinct elements of
  $[n+d]\setminus T$ and $Tij$ is short for $T\cup\{i,j\}$.  Clearly, the map $\pi$ is a
  lifting function on the hypersimplex $\Delta(d,n+d)$, and thus it induces a regular
  subdivision.  That the tropical Pl\"ucker vectors, in fact, induce matroid
  decompositions is a known fact \cite[\S1.2]{Kapranov93}, \cite[Proposition
  2.2]{Speyer04}.

  One can verify directly that the point configuration $V$ in $\TA^{d-1}$ is (in the sense
  of Theorem~\ref{thm:main}) dual to the regular subdivision induced by $\pi$ at the
  vertex figure of $\Delta(d,n+d)$ at the vertex $e_{n+1}+e_{n+2}+\dots+e_{n+d}$.
\end{proof}

Notice that our computation in the proof above makes use of ``$\infty$'' as a coordinate,
that is, from now on we work in the \emph{tropical projective space}
\[
\TP^{d-1} \ := \
((\RR\cup\{\infty\})^d\setminus\{(\infty,\infty,\dots,\infty)\})/\RR\vones_d \, .
\]
The tropical projective space $\TP^{d-1}$ is a compactification of the tropical torus
$\TA^{d-1}$ with boundary.  In fact, $\TA^{d-1}$ should be seen as an open ordinary
regular simplex of infinite size which is compactified by $\TP^{d-1}$ in the natural way.

\begin{example}\label{ex:pluecker}
  Starting out with the point configuration $V$ from Example~\ref{ex:BlockYu-Fig1:1} with
  $d=3$ and $n=4$ we obtain
  \begin{equation}\label{eq:V_bar}
    \bar V \ = \
    \begin{pmatrix}
      0 & 0 & 0 & 1 & 0 & \infty & \infty \\
      3 & 5 & 0 & 5 & \infty & 0 & \infty \\
      6 & 2 & 1 & 0 & \infty & \infty & 0
    \end{pmatrix} \, .
  \end{equation}
  The resulting tropical Pl\"ucker vector is given in Table~\ref{tab:pluecker}, and the
  induced matroid subdivision of $\Delta(d,n+d)$ is listed in Table~\ref{tab:matroid}.
  For each of the ten maximal cells the corresponding set of bases is listed.  These
  maximal cells are in bijection with the pseudo-vertices of $\tconv V$ with respect to
  $V$ as given in Tables~\ref{tab:BlockYu-Fig1:types} and~\ref{tab:triangulation}.
\end{example}

\begin{table}[hbt]
  \caption{Tropical Pl\"ucker vector $\pi$ of $\bar V$ from Equation~\eqref{eq:V_bar}.}
  \label{tab:pluecker}
  \renewcommand{\arraystretch}{0.9}
  \small
  \begin{tabular*}{\linewidth}{@{\extracolsep{\fill}}llllllllllll@{}}
    \toprule
    123:2 & 124:3 & 125:5 & 126:2 & 127:3 & 134:0 & 135:4 & 136:1 & 137:0 & 145:3 & 146:0 & 147:4 \\
    156:6 & 157:3 & 167:0 & 234:0 & 235:2 & 236:1 & 237:0 & 245:5 & 246:0 & 247:5 & 256:2 & 257:5 \\
    267:0 & 345:0 & 346:0 & 347:1 & 356:1 & 357:0 & 367:0 & 456:0 & 457:5 & 467:1 & 567:0 & \\
    \bottomrule
  \end{tabular*}
\end{table}  

\begin{table}[hbt]
  \caption{Matroid subdivision of $\Delta(3,7)$ induced by the tropical Pl\"ucker vector $\pi$
    in Table~\ref{tab:pluecker}.}
  \label{tab:matroid}
  \renewcommand{\arraystretch}{0.9}
  \begin{tabular*}{\linewidth}{@{\extracolsep{\fill}}cl@{}}
    \toprule
    label & \multicolumn{1}{c}{matroid bases} \\
    \midrule
    $v_1$ & \begin{small} 125 126 135 136 145 146 156 157 167 256 356 456 567 \end{small}\\
    $v_2$ & \begin{small} 124 125 127 145 157 234 235 237 245 246 256 257 267 345 357 456 567 \end{small} \\
    $v_3$ & \begin{small} 134 136 137 146 167 234 236 237 246 267 345 346 356 357 367 456 567 \end{small} \\
    $v_4$ & \begin{small} 124 127 145 147 157 234 237 246 247 267 345 347 357 456 457 467 567 \end{small} \\[1ex]
    $w_1$ & \begin{small} 134 137 146 167 234 237 246 267 345 346 347 357 367 456 467 567 \end{small} \\
    $w_2$ & \begin{small} 124 127 145 157 234 237 245 246 247 257 267 345 357 456 457 567 \end{small} \\
    $w_3$ & \begin{small} 123 124 125 126 127 134 137 145 146 157 167 234 235 237 246 256 267 345 357 456 567 \end{small} \\
    $w_4$ & \begin{small} 123 125 126 134 135 136 137 145 146 157 167 235 256 345 356 357 456 567 \end{small} \\
    $w_5$ & \begin{small} 124 127 134 137 145 146 147 157 167 234 237 246 267 345 347 357 456 467 567 \end{small} \\
    $w_6$ & \begin{small} 123 126 134 136 137 146 167 234 235 236 237 246 256 267 345 356 357 456 567 \end{small} \\
    \bottomrule
  \end{tabular*}
  \label{tab:BlockYu-Fig1:matroids}
\end{table}

The hypersimplex $\Delta(d,n+d)$ is the set of points
$(x_1,x_2,\dots,x_{n+d})\in\RR^{n+d}$ which satisfies the $2(n+d)$ linear inequalities
$0\le x_i\le 1$ and the linear equation $\sum x_i=d$.  The given inequalities are all
facet defining.  The facet defined by $x_i=0$ is the \emph{$i$-th deletion facet}, and the
facet $x_1=1$ is the \emph{$i$-th contraction facet} of $\Delta(d,n+d)$.  Each deletion
facet is isomorphic to $\Delta(d,n+d-1)$, and each contraction facet is isomorphic to
$\Delta(d-1,n+d-1)$.  Recursively, all faces of hypersimplices are hypersimplices.  Now
$\MatroidPolytope{\cM}$ is a subpolytope of $\Delta(d,n+d)$ whenever $\cM$ is a matroid of
rank $d$ on the set $[n+d]$, and the intersection of $\MatroidPolytope{\cM}$ with a facet
is again a matroid polytope, a \emph{deletion} or a \emph{contraction} depending on the
type of the facet.  This implies that each matroid subdivision of $\Delta(d,n+d)$ induces
a matroid subdivision on each facet.

\begin{table}[htb]
  \caption{Restriction $\pi_1$ of the tropical Pl\"ucker vector $\pi$ from
    Table~\ref{tab:pluecker} to the first contraction facet.}
  \label{tab:contraction}
  \renewcommand{\arraystretch}{0.9}
  \small
  \begin{tabular*}{\linewidth}{@{\extracolsep{\fill}}lllllllllllllll@{}}
    \toprule
    23:2 & 24:3 & 25:5 & 26:2 & 27:3 & 34:0 & 35:4 & 36:1 & 37:0 & 45:3 & 46:0 & 47:4 &
    56:6 & 57:3 & 67:0 \\
    \bottomrule
  \end{tabular*}
\end{table}

We want to look at the matroid subdivisions of $\Delta(d,n+d)$ for low rank $d$.  The
hypersimplex $\Delta(1,n+1)$ is an $n$-dimensional simplex without any non-trivial
(matroid) subdivisions.  For $d=2$ the situation is more interesting: Each matroid
subdivision of $\Delta(2,n+2)$ is dual to a tree with $n$ leaves \cite[\S1.3]{Kapranov93}.
For $d=3$ we have the following result.

\begin{theorem}[{\cite[Theorem~4.4]{HerrmannJensenJoswigSturmfels08}}]\label{thm:arrangements}
  The regular matroid subdivisions of $\Delta(3,n+3)$ bijectively correspond to the
  equivalence classes of arrangements of $n+3$ metric trees with $n+2$ labeled leaves.
\end{theorem}

It is easy to see where the $n+3$ trees come from: They are dual to the matroid
subdivisions induced on the $n+3$ contraction facets of $\Delta(3,n+3)$.  A sequence
$(\delta_1,\delta_2,\dots,\delta_n)$ is an \emph{arrangement} of metric trees if
$\delta_i$ is a tree metric on the set $[n]\setminus\{i\}$ satisfying
\[
\delta_i(j,k) \ = \ \delta_j(k,i) \ = \ \delta_k(i,j)
\]
for any three distinct $i,j,k\in[n]$.  Arrangements of metric trees are \emph{equivalent}
if they induce the same \emph{arrangement of abstract trees}, which roughly means that
their intersection patterns are the same; see \cite{HerrmannJensenJoswigSturmfels08} for
details.

\begin{example}
  Restricting the tropical Pl\"ucker vector $\pi$ on $\Delta(3,7)$ from
  Example~\ref{ex:pluecker} to the first contraction facet gives a tropical Pl\"ucker
  vector $\pi_1$ on the second hypersimplex $\Delta(2,6)$ (with coordinate directions
  labeled $2,3,\dots, 7$) shown in Table~\ref{tab:contraction}.  Since $\pi_1$ happens to
  take values between $0$ and~$6$ we can set $\delta_1(S):=3-(\pi_1(S)/6)$ for
  $S\in\tbinom{\{2,3,\dots,7\}}{2}$, which yields a metric on the set $\{2,3,\dots,7\}$.
  Written as one half of a symmetric matrix we have
  \[
  \delta_1 \ = \
  \begin{pmatrix}
    0 & 8/3 & 5/2 & 13/6 &  8/3 & 5/2 \\
    &   0 &   3 &  7/3 & 17/6 &   3 \\
    &     &   0 &  5/2 &    3 & 7/3 \\
    &     &     &    0 &    2 & 5/2 \\
    &     &     &      &    0 &   3 \\
    &     &     &      &      &   0
  \end{pmatrix} \, .
  \]
  Notice that the matroid subdivision of the first contraction facet induced by $\pi$ does
  not change if we add multiples of $(1,1,\dots,1)$ to $\pi$ or scale it by any positive
  real number.  The rescaling chosen above makes sure that $\delta_1$ takes values between
  two and three only.  This way the triangle inequality becomes valid automatically, which
  is why $\delta_1$ is indeed a metric.  Moreover, the metric $\delta_1$ is tree-like, and
  the corresponding metric tree is shown in Figure~\ref{fig:tree}.  For instance, one can
  check that $\delta_1(2,3)=8/3=13/12+1/6+17/12$.  The edge lengths on the tree can be
  computed via the Split Decomposition Theorem of Bandelt and Dress \cite{BandeltDress92}.
\end{example}

\begin{figure}[hbt]
  \includegraphics[width=.75\textwidth]{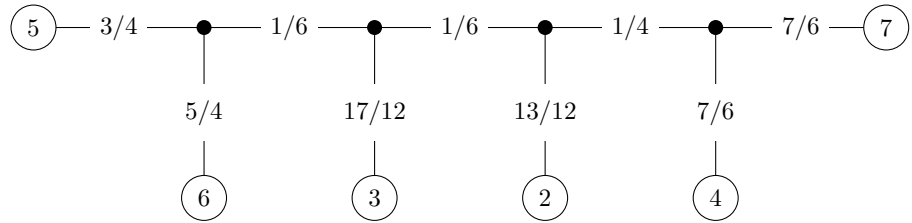}

  \caption{Caterpillar tree with six labeled leaves and edge lengths corresponding to the
    metric $\delta_1$.}
  \label{fig:tree}
\end{figure}

So far we looked at a single tree only.  Now we want to explore how the $n+3$ trees, one
for each contraction facet, interact.  In the plane $\TA^2$ a tropical line is the same as
a hyperplane.  That is to say, the tropical hyperplane/line
\[
H(a) \ = \ -a + \bigl(\RR_{\ge 0}e_1 \cup \RR_{\ge0}e_2 \cup \RR_{\ge0}e_3\bigr)
\]
is the union of the infinite rays into the three coordinate directions emanating from the
apex $-a$.  In particular, it can be seen as a metric tree with one trivalent internal node
and three edges of infinite length.  The taxa sit at the endpoints of those rays in the
compactification $\TP^2$.  The mirror image of $H(a)$ at its apex is its \emph{dual line}
\[
H^*(a) \ = \ -a - \bigl(\RR_{\ge 0}e_1 \cup \RR_{\ge0}e_2 \cup \RR_{\ge0}e_3\bigr) \, .
\]
Notice that $H^*(a)$ is a tropical line if one takes ``$\max$'' as the tropical addition
rather than ``$\min$''.  Any two tropical lines either share an infinite portion of some
ray or they meet in a single point, and the same holds for their duals.  The
\emph{arrangement of dual tropical lines} induced by a point sequence $V$ in $\TA^2$,
denoted as $H^*(-V)$, is the sequence of tropical lines with apices $v_1,v_2,\dots,v_n$.
In order to make the connection to matroid decompositions of hypersimplices we can instead
also look at the compactified version $\bar H^*(-V)$ which is the sequence of lines in
$\TP^2$ with the $n+3$ apices
$v_1,v_2,\dots,v_n,(0,\infty,\infty),(\infty,0,\infty),(\infty,\infty,0)$.  Now each dual
tropical line $\bar H^*(-v_i)$, for $1\le i\le n+3$, gives rise to a labeled tree by
recording the intersection pattern with the other trees.

\begin{figure}[htb]
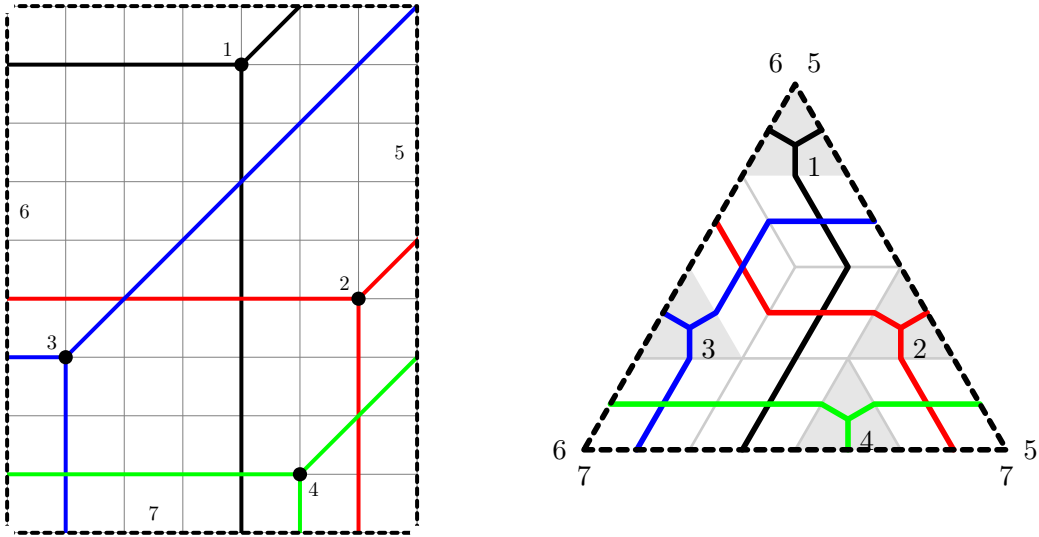

  \begin{minipage}[c]{.45\textwidth}\centering
    \includegraphics[height=\textwidth]{trop-polytope.1}
  \end{minipage}
  \quad
  \begin{minipage}[c]{.45\textwidth}\centering
    \includegraphics[height=.8\textwidth]{tree-arrangement.4}
  \end{minipage}
  
  \caption{Arrangement of seven trees, three of which cover the boundary of $\TP^2$.}
  \label{fig:trees}
\end{figure}

\begin{example}\label{ex:arrangement}
  The tree arrangement $H^*(-V)$ and its compactification $\bar H^*(-V)$ arising from our
  running example matrix $V$ introduced in Example~\ref{ex:BlockYu-Fig1:1} are shown in
  Figure~\ref{fig:trees} to the left.  The tree with apex $v_5=(0,\infty,\infty)$ occurs
  as the dashed line on the top and on the right of the rectangular section of $\TA^2$,
  the tree corresponding to $v_6=(\infty,0,\infty)$ is to the left, and the tree
  corresponding to $v_7=(\infty,\infty,0)$ is at the bottom.  For instance, the tree
  corresponding to $v_1$ receives the labeling shown in Figure~\ref{fig:tree}.  The
  subsequent Theorem~\ref{thm:orientedmatroid} can be visually verified for our example by
  comparing Figures~\ref{fig:BlockYu-Fig1} and~\ref{fig:trees}.
\end{example}

\begin{theorem}[{\cite[Theorem~6.3]{ArdilaDevelin07}}]\label{thm:orientedmatroid}
  Suppose that $V$ is a sufficiently generic sequence of points in $\TA^2$.  Then the
  polyhedral subdivision of $\TA^2$ induced by the tree arrangement $H^*(-V)$ coincides
  with the type decomposition induced by $V$.
\end{theorem}

The abstract arrangement of trees induced by the matroid decomposition of $\Delta(3,n+3)$
which in turn is induced by the point sequence $V$ is precisely the compactified tree
arrangement $\bar H^*(-V)$.

\begin{remark}
  The \emph{Cayley trick} from polyhedral combinatorics allows to view triangulations of
  $\Delta_{n-1}\times\Delta_{d-1}$ as \emph{lozenge tilings} of the dilated simplex
  $n\Delta_{d-1}$; see \cite{Santos05}.  For our Example~\ref{ex:arrangement} the
  corresponding lozenge tiling of $4\Delta_2$ is shown in Figure~\ref{fig:trees} to the
  right.  The tree arrangement $H^*(-V)$ partitions the dual graph of the tiling.
\end{remark}

\section{Computational Experiments with \polymake}

Most algorithms mentioned above are implemented in \polymake, version 2.9.4, which can be
downloaded from \url{www.polymake.de}.  After starting the program on the command line
(with the application \texttt{tropical} as default) an interactive shell starts which
receives commands in a Perl dialect.  For instance, to visualize our running
Example~\ref{ex:BlockYu-Fig1:1} it suffices to say:

\begin{small}
\begin{verbatim}
> $p = new TropicalPolytope<Rational>(POINTS=>[[0,3,6],[0,5,2],[0,0,1],[1,5,0]]);
> $p->VISUAL;
\end{verbatim}
\end{small}

We want to show how \polymake can be used for the investigation of tropical polytopes.
For instance we can compute the type of a cell in the type decomposition as follows.  This
is not a standard function, which is why it requires a sequence of commands.  Again the
pseudo-vertices and the types always refer to the fixed sequence of generators that we
started out with above.

\begin{small}
\begin{verbatim}
> print @{rows_labeled($p->PSEUDOVERTICES)};
0:0 5 0
1:0 4 -1
2:0 3 2
3:0 3 4
4:0 3 -1
5:0 0 1
6:0 0 -1
7:0 1 2
8:0 5 2
9:0 3 6

> print $p->ENVELOPE->BOUNDED_COMPLEX;
{0 1 2 4 8}
{2 3 7}
{2 4 5 6 7}
{3 9}
\end{verbatim}
\end{small}

The first command lists the pseudo-vertices (in no particular order), and the function
\texttt{rows\_labeled} is responsible for making the implicit numbering explicit which
arises from the order.  The next command lists all the bounded cells of the type
decomposition.  Each row corresponds to one maximal cell, namely the (tropical or
ordinary) convex hull of the pseudo-vertices with the given indices.  We obtain two
pentagons, one triangle, and one edge.  Now we want to compute the type of the first cell
(numbered `0'):

\begin{small}
\begin{verbatim}
> $indices = $p->ENVELOPE->BOUNDED_COMPLEX->[0];
\end{verbatim}
\end{small}  

This is just the sequence of indices of the pseudo-vertices defining this cell.  In the
following we produce an ordinary polytope which has these points as its vertices.  One
minor technical complication arises from the fact that the leading `0' used for coordinate
homogenization in the tropical world must be replaced by a `1' which makes the ordinary
polytope homogeneous.

\begin{small}
\begin{verbatim}
> $vertices = $p->PSEUDOVERTICES->minor($indices,range(1,$p->AMBIENT_DIM));
> $n_vertices = scalar(@{$vertices});
> $all_ones = new Vector<Rational>([ (1)x$n_vertices ]);
> $cell = new Polytope<Rational>(VERTICES => ($all_ones|$vertices));
> print $cell->VERTICES;
1 5 0
1 4 -1
1 3 2
1 3 -1
1 5 2
\end{verbatim}
\end{small}  

Now the type to be computed is just the type of any relatively interior point of this
ordinary polytope called \texttt{\$cell}.

\begin{small}
\begin{verbatim}
> $point = $cell->REL_INT_POINT;
> print $point;
1 4 2/5
\end{verbatim}
\end{small}  

This point must be translated back into the tropical world, and then we can call a
function to compute its type, which turns out to be $(2,13,4)$.  Notice that \polymake
starts the numbering from `0', and hence each index in the output is shifted by one.

\begin{small}
\begin{verbatim}
> $cell_point = poly2trop(new Polytope<Rational>(POINTS=>[$point]))->POINTS;
> print types($cell_point,$p->POINTS);
{1}
{0 2}
{3}
\end{verbatim}
\end{small}

It should be mentioned that there is also the \texttt{Maxplus toolbox for Scilab} for
computations in tropical convexity \cite{MaxplusToolboxScilab}.

\bibliographystyle{amsplain}
\bibliography{main}

\end{document}